\documentclass[12pt]{article}
\input{amssym}
\newtheorem{prethm}{{\bf Theorem}}

\newenvironment{thm}{\begin{prethm}{\hspace{-0.5
em}{\bf.}}}{\end{prethm}}
\newtheorem{precor}{{\bf Corollary}}

\newenvironment{cor}{\begin{precor}{\hspace{-0.5
em}{\bf.}}}{\end{precor}}
\newtheorem{preprop}{{\bf Proposition}}

\newenvironment{prop}{\begin{preprop}{\hspace{-0.5
em}{\bf.}}}{\end{preprop}}
\newtheorem{preque}{{\bf Question}}

\newtheorem{prerem}{{\bf Remark}}

\newenvironment{rem}{\begin{prerem}{\hspace{-0.5
em}{\bf.}}}{\end{prerem}}
\newtheorem{preques}{{\bf Question}}

\newtheorem{prelemma}{{\bf Lemma}}

\newenvironment{lemma}{\begin{prelemma}{\hspace{-0.5
em}{\bf.}}}{\end{prelemma}}
\newtheorem{prefact}{{\bf Fact}}

\newtheorem{preobs}{{\bf Observation}}

\newtheorem{prefig}{{\bf Figure}}

\newtheorem{prelemm}{{\bf Lemma}}

\newtheorem{preex}{{\bf Example}}

\newtheorem{prepro}{{\bf Proposition}}

\newtheorem{prelem}{{\bf Theorem}}

\newenvironment{lem}{\begin{prelem}{\hspace{-0.5
em}{\bf.}}}{\end{prelem}}
\newtheorem{preproof}{{\bf Proof.}}

\newenvironment{proof}[1]{\begin{preproof}{\rm
               #1}\hfill{$\rule{2mm}{2mm}$}}{\end{preproof}}

\newtheorem{preconj}{{\bf Conjecture}}

\newtheorem{predeff}{{\bf Definition}}

\newenvironment{deff}{\begin{predeff}{\hspace{-0.85
em}{\bf.}}}{\end{predeff}}
\def\newpic#1{}
\date{}
\begin{document}

\title{
{\Large{\bf Uniquely dimensional graphs}}}
%

{\small
\author{
{\sc Behrooz Bagheri Gh.}, {\sc Mohsen Jannesari},
{\sc Behnaz Omoomi}\\
[1mm]
{\small \it  Department of Mathematical Sciences}\\
{\small \it  Isfahan University of Technology} \\
{\small \it 84156-83111, \ Isfahan, Iran}}




 \maketitle \baselineskip15truept

\begin{abstract}
 A set $W\subseteq V(G)$ is called a  resolving set, if
for each two distinct vertices $u,v\in V(G)$ there exists $w\in W$
such that $d(u,w)\neq d(v,w)$,  where $d(x,y)$ is the distance
between the vertices $x$ and $y$. A resolving set for $G$ with
minimum cardinality is  called a  metric basis.
 A  graph with a unique metric
basis is called  a uniquely dimensional graph.
 In  this paper, we study some  properties  of uniquely dimensional graphs.
\\
\end{abstract}

\section{Introduction}
 Throughout the paper, $G=(V,E)$ is a
finite, simple, and connected  graph of order $n$. The distance
between two vertices $u$ and $v$, denoted by $d(u,v)$, is the
length of a shortest path between $u$ and $v$ in $G$. For a vertex
$v\in V(G)$, $\Gamma_i(v)=\{u \ |\ d(u,v)=i\}$. The diameter of
$G$ is ${\rm diam(G)}=\max\{d(u,v)\ |\ u,v\in V(G)\}$.   The girth
of $G$ is the length of a shortest cycle in~$G$.    The set of
all adjacent vertices to a vertex $v$ is denoted by $N(v)$ and
$|N(v)|$ is  the degree of a vertex $v$, $\deg(v)$. The maximum
degree and the minimum degree of a graph $G$, are denoted by
$\Delta(G)$ and $\delta(G)$, respectively. The notations $u\sim
v$ and $u\nsim v$ denote the adjacency and non-adjacency
relations between $u$ and $v$, respectively.

For an ordered set $W=\{w_1,w_2,\ldots,w_k\}\subseteq V(G)$ and a
vertex $v$ of $G$, the  $k$-vector
$$r(v|W)=(d(v,w_1),d(v,w_2),\ldots,d(v,w_k))$$
is called  the {\it metric representation}  of $v$ with respect
to $W$. The set $W$ is called a {\it resolving set} for $G$ if
distinct vertices have different metric representations.  A resolving
set for $G$ with minimum cardinality is  called a {\it metric
basis}, and its cardinality is the {\it metric dimension} of $G$,
denoted by $\beta(G)$. If $\beta(G)=k$, then $G$ is said to be
$k$-dimensional.

 In~\cite{Slater1975}, Slater introduced the idea of a resolving
set and used a {\it locating set} and the {\it location number}
for what we call a resolving set and the metric dimension,
respectively. He described the usefulness of these concepts when
working with U.S. Sonar and Coast Guard Loran stations.
Independently, Harary and Melter~\cite{Harary} discovered the
concept of the location number as well and called it the metric
dimension. For more results related to these concepts
see~\cite{cartesian product,bounds,sur1,landmarks}. The concept
of a resolving set has various applications in diverse areas
including coin weighing problems~\cite{coin}, network discovery
and verification~\cite{net2}, robot navigation~\cite{landmarks},
mastermind game~\cite{cartesian product}, problems of pattern
recognition and image processing~\cite{digital}, and
combinatorial search and optimization~\cite{coin}.

It is obvious that to see  whether a given set $W$ is a resolving
set, it is sufficient to consider the
 vertices in $V(G)\backslash W$, because
 $w\in W$ is the unique vertex in $G$ for which $d(w,w)=0$.
 When $W$ is a resolving set for $G$, we say
that $W$ {\it resolves} $G$. In general, we say an ordered set $W$
resolves a set $T\subseteq V(G)$, if for each two distinct
vertices $u,v\in T$, $r(u|W)\neq r(v|W)$.

The following bound is the known upper bound for the metric
dimension.
\begin{lem}~{\rm\cite{Ollerman}}\label{thm:B<n-d}
If $G$ is a connected graph of order $n$ and diameter $d$, then
$\beta(G)\leq n-d$.
\end{lem}

In \cite{randomly1, randomly2}, the properties of $k$-dimensional
graphs  in which every $k$ subset of vertices is a metric basis
are studied. Such graphs are called randomly $k$-dimensional
graphs. In the opposite point  there are  graphs which have a
unique metric basis.

\begin{deff}
A graph $G$ is called  {\rm uniquely dimensional} if $G$ has a
unique metric basis. A uniquely dimensional graph $G$ with
$\beta(G)=k$ is called   a {\rm uniquely $k$-dimensional} graph.
\end{deff}

In this paper, we  first obtain some upper bounds for the metric
dimension of uniquely dimensional graphs. Then, we give some
construction for uniquely $k$-dimensional graphs of the given
order. Finally, we obtain  a lower bound and an upper bound for
the minimum order of uniquely $k$-dimensional graphs in terms of
$k$.
\section{Some upper bounds }\label{main}

In this section we obtain  some upper bounds for the metric
dimension of uniquely dimensional graphs.

Two vertices $u,v\in V(G)$ are called {\it  twin} vertices if
$N(u)\setminus\{ v\}=N(v)\setminus \{u\}$. It is known that, if
$u$ and $v$ are twin vertices, then every resolving set $W$ for
$G$ contains at least one of the vertices $u$ and $v$. Moreover,
if $u\notin W$ then $(W\setminus v)\cup \{u\}$ is also a
resolving set for $G$. \cite{extermal}

 For a uniquely dimensional graph
we have the following fact.
\begin{lemma}\label{lem:notwin} If $G$ is a uniquely dimensional graph,
then  $G$ contains no twin vertices.
\end{lemma}
\begin{proof}{
Let $B$ be the unique metric basis of $G$. If $u,v\in V(G)$ are twin vertices,
 then $u,v\in B$; otherwise we can replace the one in $B$ with the other one.
Now,  since $B\setminus \{u\}$ is not a basis of $G$, there is
exactly one vertex
 $w\in V(G)\setminus B$ such that $r(u|B\setminus\{u\})=r(w|B\setminus\{u\})$.
 Consequently, $(B\setminus\{u\})\cup\{w\}$ is a metric basis of $G$  different from $B$, which is a contradiction.
 }\end{proof}

\begin{thm}\label{thm:k<n-d-2}
If $G$ is a uniquely dimensional graph of order $n$ and   diameter
$d$, then $\beta(G)\leq n-d-2$.
\end{thm}

\begin{proof}{
Let $(v_0,v_1,\ldots,v_d)$ be a path of length $d$ in $G$. Two
sets  $V(G)\setminus\{v_1,v_2,\ldots,v_d\}$ and
$V(G)\setminus\{v_0,v_1,\ldots,v_{d-1}\}$ are two resolving set
of $G$ of size $n-d$. Hence, if $G$ is  uniquely dimensional, then
$\beta(G)\leq n-d-1$.
  To complete the proof
we show that $\beta(G)\neq n-d-1$.

Let $\beta(G)=n-d-1$ and for each $i$, $1\leq i\leq d$,
$\Gamma_i=\Gamma_i(v_0)$. We claim that for each $i$, $1\leq i\leq
d$, $\Gamma_i$ is an independent set or a clique; otherwise there
exists an $i$ for which $\Gamma_i$ contains vertices $x,y,z$ such
that $x\sim y$ and $x\nsim z$. Therefore,
$V(G)\setminus\{y,z,v_1,v_2,\ldots,v_{i-1},v_{i+1},\ldots,v_d\}$
is a metric basis of $G$.  Now, if $y\nsim z$, then
$V(G)\setminus\{x,z,v_1,v_2,\ldots,v_{i-1},v_{i+1},\ldots,v_d\}$
 and if $y\sim z$,  then
$V(G)\setminus\{x,y,v_1,v_2,\ldots,v_{i-1},v_{i+1},\ldots,v_d\}$
is another metric basis of $G$, respectively,  which both are
contradictions. Consequently, for each $i$, $1\leq i\leq d$,
$\Gamma_i$ is an independent set or a clique.

Now let for some $i,~1\leq i\leq d$, $|\Gamma_i|\geq 2$. Then, all
vertices in $\Gamma_i$ are adjacent to all vertices in
$\Gamma_{i-1}$; otherwise there exist $a\in  \Gamma_{i-1}$ and
$x\in \Gamma_i$ such that $a\nsim x$. Therefore, $x$ has a
neighbor in $\Gamma_{i-1}$,  say $b$. Assume that $y\in
\Gamma_i$  and $y\neq x$. Clearly $i \geq 2$. Thus,
$V(G)\setminus\{a,b,y,v_1,v_2,\ldots,v_{i-2},v_{i+1},\ldots,v_d\}$
is a metric basis of $G$. Now, if $y\sim a$, then
$V(G)\setminus\{b,x,y,v_1,v_2,\ldots,v_{i-2},v_{i+1},\ldots,v_d\}$,
and if $y\nsim b$, then
$V(G)\setminus\{a,x,y,v_1,v_2,\ldots,v_{i-2},v_{i+1},\ldots,v_d\}$
is another metric basis of $G$, respectively. These contradictions
imply that $y\nsim a$ and $y\sim b$. Hence,
$V(G)\setminus\{a,b,x,v_1,v_2,\ldots,v_{i-2},v_{i+1},\ldots,v_d\}$
is  a metric basis of $G$, which is also a contradiction.
Consequently, all vertices in $\Gamma_i$ are adjacent to all
vertices in $\Gamma_{i-1}$.

The above two facts imply that,   if $|\Gamma_i|\geq2$ and
$|\Gamma_{i+1}|\geq2$, then all vertices in $\Gamma_i$ have the
same neighbors in $\Gamma_{i-1}\cup\Gamma_i\cup\Gamma_{i+1}$.
Therefore,  all vertices $u,v\in\Gamma_i$ are twin vertices, which
 by Lemma~\ref{lem:notwin} this is impossible. Thus,
$|\Gamma_i|\geq2$ implies that $|\Gamma_{i+1}|=1$ and
$|\Gamma_{i-1}|=1$. Hence, if $|\Gamma_i|>2$, then since $\Gamma_{i+1}=\{v_{i+1}\}$,   by the
Pigenhole principle there are  two vertices $u,v\in \Gamma_i$ with
the same adjacency relation with $v_{i+1}$ .  Therefore, $u$
and $v$ are twin vertices, which is impossible.
 That is, for each $i,~1\leq i\leq d$,
$|\Gamma_i|\leq2$. Now let $j$ be the largest integer in
$\{1,2,\dots,d\}$ with $|\Gamma_j|=2$ and $\Gamma_j=\{v_j,y_j\}$,
where $y_j$ is the vertex with no neighbor in $\Gamma_{j+1}$.
Therefore, the sets $\{v_0,v_d\}$ and $\{v_0,y_j\}$ are two metric
bases of $G$. This contradiction implies that $\beta(G)\neq
n-d-1$.}\end{proof}

\begin{thm}\label{girth}
If $G$ is a uniquely dimensional graph of order $n$ and girth $g$, then $\beta(G)\leq n-g+1$.
\end{thm}

\begin{proof}{
Let  $C_g=(v_1,v_2,\ldots,v_g,v_1)$ be a shortest cycle in $G$.
Then $V(G)\setminus\{v_3,v_4,\ldots,v_g\}$ and
$V(G)\setminus\{v_2,v_3,\ldots,v_{g-1}\}$ are two resolving set
for $G$ of size $n-g+2$. Since $G$ has a unique basis, none of
these two sets is a metric basis of $G$. Therefore, $\beta(G)\leq
n-g+1$. }\end{proof}

\begin{thm}\label{thm:n>2k} If $G$ is a uniquely dimensional graph of order $n$,
then $\beta(G)< {n \over 2}$.
\end{thm}

\begin{proof}{
By the  contrary assume that  $G$ has a unique metric basis
$B=\{v_1,v_2,\ldots,v_k\}$ and $n\leq 2k$.
 Since $k\leq n-1$, $W=(V(G)\setminus B)\cup\{v_1,v_2,\ldots, v_{2k-n}\}\neq
 B$ with $|W|=k$.
 Therefore,  $W$ is not a basis of $G$ and there exist vertices
$x,y\in V(G)\setminus W\subseteq B$ such that $r(x|W)=r(y|W)$.
Say $x=v_i$ and $y=v_j$. Hence, for each $v\in V(G)\setminus B$,
$d(v,v_i)=d(v,v_j)$. By this reason, $B\setminus\{v_i\}$ resolves
$V(G)\setminus B$. Therefore, there is exactly one vertex $u\in
V(G)\setminus B$ such that
$r(u|B\setminus\{v_i\})=r(v_i|B\setminus\{v_i\})$. Consequently,
$(B\setminus\{v_i\})\cup\{u\}$ is a metric basis of $G$, which is
a contradiction. Thus, $2\beta(G)< n$.
}\end{proof}

\section{ Construction of uniquely $k$-dimensional graphs}\label{recursive}
In this section, we provide some  construction for uniquely
$k$-dimensional graphs of given order. Then we end with giving a
lower bound and an upper bound for the minimum number of vertices
in such graphs in terms of $k$.

\begin{rem}\label{remark}
Note that, if $G$ is a graph of
diameter $d$, then every $W\subseteq V(G)$ can resolve at most $d^{|W|}$ vertices of $V(G)\setminus W$.
Hence, every $k$-dimensional graph of diameter $d$ has at most $k+d^k$ vertices.
\end{rem}

In~\cite{unique}, Buczkowski et al. constructed a uniquely
$k$-dimensional graph   with diameter $2$ and order $k+2^k$.

\begin{lem}~{\rm\cite{unique}}\label{kunique}
For $k\ge 2$, there exists a uniquely $k$-dimensional graph of
order $n=k+2^k$,  diameter $2$,  and maximum degree $n-1$.
\end{lem}

In the following theorem regarding to constructing uniquely
$k$-dimensional graphs with diameter $d$,  we obtain two necessary
conditions for the existence of  $k$-dimensional graphs with
diameter $d$ and  order $k+d^k$.

\begin{thm}\label{lem:Gammai in k+d^k}
If $G$ is a  $k$-dimensional graph with diameter $d$ and
order $k+d^k$,  then \\
{\rm (i)} $d\leq 3$.\\
 {\rm (ii)}  For a  basis $B$ and  every $v\in B$,
 $|\Gamma_d(v)|\geq d^{k-1}$.
\end{thm}
\begin{proof}{
(i) Let $G$ be a $k$-dimensional graph of  diameter $d \geq 4$
and  order $k+d^k$. Thus,   $V(G)=U\cup B$, where
$U=\{u_1,u_2,\ldots,u_{d^k}\}$ and the ordered set
$B=\{v_1,v_2,\ldots,v_{k}\}$ is a basis of~$G$. Clearly,
$\{r(u_i|B)\ | \ 1\le i\le d^k\}=[d]^k$, where $[d]^k$ denotes the
set of all $k$-tuples with entries in $\{1,2,\ldots,d\}$. Without
loss of generality, suppose that $r(u_1|B)=(1,1,\ldots,1)$ and
$r(u_2|B)=(4,1,\ldots,1)$. Therefore, $d(v_1,v_2)\leq 2$ and
$d(u_2,v_1)\leq d(u_2,v_2)+d(v_2,v_1)\leq 3$, a contradiction.
Thus, $d\leq 3$.

\noindent (ii) Let $B=\{v_1,v_2,\ldots,v_k\}$. By the order and
diameter of $G$, each $k$-vector with coordinates in
$\{1,2,\ldots,d\}$ is the metric representation of a vertex $u\in
V(G)\setminus B$ with respect to $B$.
 Therefore, for each $v\in B$,
there are $d^{k-1}$ vertices of $G$ that the $i$-th coordinate of
their metric representations is $d$. Thus, $|\Gamma_d(v)|\geq
d^{k-1}$.
}\end{proof}

 In the following, we give a construction
for uniquely $k$-dimensional graphs of diameter $3$ and order
$k+3^k$.

\begin{thm}\label{pro:n=k+3^k} For every  integer $k\geq2$, there exists
a uniquely $k$-dimensional graph of diameter $3$ and  order
$k+3^k$.
\end{thm}

\begin{proof}{Let $G$ be a graph with vertex set $U\cup W$, where
$U=\{u_1,u_2,\ldots,u_k\}$ is an independent set and $W$ is the
set of all $k$-tuples with entries in $\{1,2,3\}$ and  two
vertices $x,y\in W$ are adjacent if they are different in exactly
one coordinate and this difference is one. Moreover, the vertex
 $(2,2,\ldots,2)$ is adjacent
to all vertices in $W$. Also, $w\in W$ is adjacent to $u_i\in U$
if the $i$-th coordinate of $w$ is $1$.

 The vertex $(2,2,\ldots,2)$ is adjacent to all
vertices in $W$ and $(1,1,\ldots,1)$ is adjacent to all vertices
in $U$, thus ${\rm diam}(G)\leq 3$. On the other hand,
$d((3,3,\ldots,3),u_1)=3$. Therefore, ${\rm diam}(G)=3$. Since
${\rm diam}(G)=3$ and the order of $G$ is $k+3^k$, by
Remark~\ref{remark}, $\beta(G)\ge k$. For each $w\in W$,
$r(w|U)=w$, thus, $U$ is a resolving set for $G$ of size $k$.
Hence, $U$ is a metric basis of $G$.

 Now since
${\rm diam}(\langle W\rangle)=2$,  for each $w\in W$,
$|\Gamma_1(w)\cup\Gamma_2(w)|\geq 3^k-1$ and hence
$|\Gamma_3(w)|\leq k<3^{k-1}$. Therefore, by
Theorem~\ref{lem:Gammai in k+d^k}(ii), no vertex of $W$ is in a
metric basis of $G$. Consequently, $U$ is the unique metric basis
of $G$.}\end{proof}

 By Theorems~\ref{thm:k<n-d-2} and~\ref{thm:n>2k}, if
$G$ is a uniquely $k$-dimensional graph of order $n$, then $n\ge
k+d+2$ and $n\ge 2k+1$. Let  $$n_0(k)=\min\{n\ |\ \mbox{ there
exists a uniquely $k$-dimensional graph of order $n$}\}.$$ Hence,
we have $\max\{2k+1, k+d+2\} \le n_0(k)$.

 The following theorem shows that if a uniquely
$k$-dimensional graph of order $n_0$ exists, then for every
$n\geq n_0$, a uniquely $k$-dimensional graph of order $n$ exists.

\begin{thm}
If $G$ is a uniquely $k$-dimensional graph of order $n_0$, then
for every $n\ge n_0$, there exists a uniquely  $k$-dimensional
graph of order $n$.
\end{thm}
\begin{proof}{
Let $G$ be a given uniquely $k$-dimensional graph of order $n_0$
and $u$ be a vertex in the basis $B$. Assume that $v_0\in
V(G)\setminus B$ is a vertex that $d(v_0,u)=\max \{d(v,u)\ |\
v\in V(G)\setminus B\}$. We construct a graph $G'$ by identifying
an end vertex of a path $P$ of length $n-n_0$ by $v_0$. By the
property of $v_0$, $B$ is also a resolving set for $G'$. Thus,
$\beta(G')\le k$. On the other hand, since every basis of $G'$
contains at most one vertex of the path $P$,  by replacing that
vertex by $v_0$, we obtain a basis for  $G$. Thus, $G'$ is also a
uniquely $k$-dimensional graph.
}\end{proof}

In the following theorem we give  a recursive construction for
uniquely dimensional graphs to obtain an upper bound for $n_0(G)$.

\begin{thm}\label{pro:recursive}
If $G_i$, $i=1,2$, is a uniquely $k_i$-dimensional graph of order
$n_i$ with $\Delta(G_i)=n_i-1$, then there exists a uniquely
$(k_1+k_2)$-dimensional graph $G$ of order $n_1+n_2-1$ with
$\Delta(G)=n_1+n_2-2$.
\end{thm}
\begin{proof}{
 Let $G_i$  be a uniquely $k_i$-dimensional graph of order
$n_i$  with the  basis $B_i$  and  $v_i\in V(G_i)$ such that
$\deg(v_i)=n_i-1$, for $i=1,2$. Let $G$ be a graph that obtained
from joining $G_1$ and $G_2$, and then identifying $v_1$ and
$v_2$,  say $v_0$. Thus, $\deg(v_0)=n_1+n_2-2$.
 Since for every $u\in V(G_1)\setminus \{v_1\}$ and $v\in
V(G_2)\setminus \{v_2\}$,  $d(u,v)=1$, if  $B$ is a basis of $G$,
then $B\cap V(G_i)$ is a basis of $G_i$, for $i=1,2$.  Therefore,
$B$ is the unique basis of $G$.}\end{proof}

\begin{prop}\label{pro:9,3}
There exists a uniquely $3$-dimensional graph of order $9$ and
maximum degree~$8$.
\end{prop}

\begin{proof}{
Let $U=\{u_1, u_2, u_3\}$ and $W=\{w_1, w_2,\ldots, w_6\}$. Also
let $G$ be graph with $V(G)=U\cup W$ and $E(G)=\{w_iw_j\ |\ 1\le
i\ne j\le 6\}\cup\{u_iw_j\ |\ 1\le i\le 3,  j=i,i+1,6\}$. We show
that $U$ is the unique basis of $G$.

 Clearly, ${\rm diam}(G)=2$. Since
$|V(G)|=9$, by Remark~\ref{remark}, $\beta(G) \geq 3$. It is easy
to see that $U$ is  resolving set and consequently  is a basis of
$G$. Now let $B$ be another basis of $G$. Since $\langle W
\rangle$ is a complete graph, $B \nsubseteq W$. Therefore,
$|B\cap W|=1$ or $2$. If $|B\cap W|=1$, then five vertices of $W$
have the same representation with respect to $B\cap W$ while
since ${\rm diam}(G)=2$,  $B\setminus W$ can not resolve five
vertices. If $|B\cap W|=2$, then four vertices of $W$ have the
same representation with respect to $B\cap W$ while  $B\setminus
W$ can not resolve $4$ vertices. These contradictions imply that
$U$ is the unique basis of $G$.}\end{proof}

In the following theorem, based on the recursive construction in
Theorem~\ref{pro:recursive}, we obtain an upper bound for
$n_0(k)$.

\begin{thm}\label{thm:5k'}
For every $k$, $k\geq2$,  there exists a uniquely $k$-dimensional
graph of order $\lceil{5k\over2}+1\rceil$.
\end{thm}

\begin{proof}{
Let $k$ be a positive integer.  If $k=2k'$, then the graph $G$
obtained  by the recursive construction given  in
Theorm~\ref{pro:recursive} from $k'$ copies of the uniquely
$2$-dimensional graph of order $6$, constructed in
Theorem~\ref{kunique} is a uniquely $k$-dimensional graph of
order $6k'-(k'-1)=5k'+1={5k\over2}+1$.

If $k=2k'+1$, then the graph $G$ obtained  by the recursive
construction  given in Theorem~\ref{pro:recursive} from $k'-1$
copies of the uniquely $2$-dimensional graph of order $6$,
constructed in Theorem~\ref{kunique}  and one copy of the
uniquely $3$-dimensional graph of order $9$ given in
Proposition~\ref{pro:9,3}, is a uniquely $k$-dimensional graph of
order $6(k'-1)-(k'-2)+8=5k'+4=\lceil{5k\over2}+1\rceil$.
}\end{proof}

Although the above theorem provides the  recursive construction for uniquely $k$-dimentional graphs of order
 $\lceil{5k\over2}+1\rceil$, to get the more explicit construction,      we construct uniquely $k$-dimensional graphs of order $3k$, in the following theorem.

\begin{thm}\label{prop:3n0<n} For each $k\geq2$, there exists
a uniquely $k$-dimensional graph of order $3k$.
\end{thm}

\begin{proof}{
Let $U=\{u_1,u_2,\ldots,u_k\}$ and $W=\{w_1,w_2,\ldots,w_{2k}\}$.
Also, let $G$ be a graph with vertex set $V(G)=U\cup W$ such that
the induced subgraph $\langle W\rangle$ of $G$ be a complete
graph, $U$ be an independent set, $u_k$ be adjacent to $w_{2i}$,
$1\leq i\leq k$, and for each $i$, $1\leq i\leq k-1$, $u_i$ be
adjacent to $w_{2i-1}$ and $w_{2i}$. We prove that $G$ is the
desired graph.

Let $w_i$ and $w_j$ be two arbitrary vertices of $V(G)\setminus
U=W$. If
 $i$ and $j$ have different  parity, then $d(w_i,u_k)\neq d(w_j,u_k)$. If
$i$ and $j$ have the same  parity, then $\lfloor {i\over
2}\rfloor\neq\lfloor {j\over 2}\rfloor $ and hence $d(w_i,u_i)\neq
d(w_j,u_i)$. Therefore, $U$ is a resolving set for $G$ of size
$k$ and  $\beta(G)\leq k$.

Now let $B$ be  a metric basis of $G$. If $u_k\notin B$, then to
resolve the set $\{u_1,w_1,w_2,w_{2k-1},w_{2k}\}$, $B$ should
contain at least three vertices from this set, since
 $\langle W\rangle$ is
a complete graph, while  replacing these three vertices by $u_1$
and $u_k$ provides a resolving set with smaller size. This
contradiction implies that $u_k\in B$. If for some $i,~1\leq
i\leq k-1$, $u_i\notin B$,  then to resolve the set
$\{u_i,w_{2i-1},w_{2i},w_{2k-1},w_{2k}\}$, $B$ should contain at
least two vertices from $\{w_{2i-1},w_{2i},w_{2k-1},w_{2k}\}$,
because $\langle W\rangle$ is a complete graph. But replacing
these two vertices by $u_i$ provides a resolving set with smaller
size. This contradiction implies that $U\subseteq B$. Since $U$
is a resolving set, $U=B$ is the unique metric basis of $G$.
}\end{proof}

By Theorems~\ref{thm:n>2k} and~\ref{thm:5k'}, we have the following corollary.

\begin{cor} Let $k\ge 2$ be an integer. Then
$2k+1\le n_0(k)\le \lceil{5k\over2}+1\rceil$.
\end{cor}

For $k=2$,  $n\ge 4+d$ implies $n\ge 6$. Hence,  $n_0(2)=6$.
 It can be seen, there is no uniquely $3$-dimensional
graph of order $7$. Thus, $8\le n_0(3)\le 9$. The determination of
$n_0(k)$, for every integer $k$ could be  an nontrivial
interesting problem.

\end{document}